\documentclass[12 pt]{amsart}
 
\pdfoutput=1
 
\usepackage{amsthm,amsmath,amssymb,enumerate,mathtools}
\newtheorem{theorem}{Theorem}[section]
\newtheorem{claim}{}[theorem]
\newtheorem{lemma}[theorem]{Lemma}

\newtheorem{proposition}[theorem]{Proposition}

\newtheorem{conjecture}[theorem]{Conjecture}
\theoremstyle{definition}
\newtheorem{definition}[theorem]{Definition}

\newcommand{\bF}{\mathbb F}

\newcommand{\bZ}{\mathbb Z}

\newcommand{\cA}{\mathcal{A}}
\newcommand{\cB}{\mathcal{B}}
\newcommand{\cC}{\mathcal{C}}

\newcommand{\cI}{\mathcal{I}}
\newcommand{\cH}{\mathcal{H}}
\newcommand{\cM}{\mathcal{M}}

\newcommand{\cX}{\mathcal{X}}

\newcommand{\ep}{\epsilon}

\DeclareMathOperator{\cl}{cl}

\DeclareMathOperator{\PG}{PG}

\DeclareMathOperator{\LG}{LG}

\DeclareMathOperator{\GF}{GF}

\newcommand{\del}{\!\setminus\!}

\usepackage{comment}

\title[A new matroid lift construction]{A new matroid lift construction and an application to group-labeled graphs}
\author{Zach Walsh}
\address{Mathematics Department, Louisiana State University, Baton Rouge, Louisiana, USA}
\email{walsh@lsu.edu}

\date{\today}

\begin{document}

\begin{abstract} 
A well-known result of Brylawski constructs an elementary lift of a matroid $M$ from a linear class of circuits of $M$.
We generalize this result by showing how to construct a rank-$k$ lift of $M$ from a rank-$k$ matroid on the set of circuits of $M$.
We conjecture that every lift of $M$ arises via this construction.

We then apply this result to group-labeled graphs, generalizing a construction of Zaslavsky.
Given a graph $G$ with edges labeled by a group, Zaslavsky's lift matroid $K$ is an elementary lift of the graphic matroid $M(G)$ that respects the group-labeling; specifically, the cycles of $G$ that are circuits of $K$ coincide with the cycles that are balanced with respect to the group-labeling.
For $k \ge 2$, when does there exist a rank-$k$ lift of $M(G)$ that respects the group-labeling in this same sense?
For abelian groups, we show that such a matroid exists if and only if the group is isomorphic to the additive group of a non-prime finite field. 
\end{abstract}

\maketitle


\begin{section}{Introduction}
This paper is concerned with using structure on the circuits of a matroid $M$ to construct a new matroid on the ground set of $M$.
For example, a collection $\cC$ of circuits of a matroid $M$ is a \emph{linear class} if, whenever $C_1$ and $C_2$ are circuits in $\cC$ so that $|C_1\cup C_2|-r_M(C_1\cup C_2)=2$, then every circuit $C$ of $M$ contained in $C_1\cup C_2$ is also in $\cC$.
Brylawski \cite[Prop. 7.4.15]{Brylawski} applied the dual of a construction of Crapo \cite{Crapo} to show that, for each linear class $\cC$ of circuits of a matroid $M$, one can construct a matroid $M'$ on $E(M)$ with rank at most one greater than $r(M)$.
Moreover, this matroid $M'$ is an \emph{elementary lift} of $M$, so there is a matroid $P$ with ground set $E(M) \cup \{e\}$ so that $P \del e = M'$ and $P / e = M$.

\begin{theorem}[Brylawski '86] \label{brylawski}
Let $M$ be a matroid, and let $\cC$ be a linear class of circuits of $M$. 
Then the function $r_{M'}$ defined, for all $X \subseteq E(M)$, by 
$$r_{M'}(X)=\begin{cases}
r_M(X) & \text{ if each circuit of $M|X$ is in $\cC$},  \\
r_M(X)+1 & \text{ otherwise} 
\end{cases}$$
is the rank function of an elementary lift $M'$ of $M$.
Moreover, every elementary lift of $M$ can be obtained in this way.
\end{theorem}

Note that $\cC$ is precisely the set of circuits of $M$ that are also circuits of $M'$.
Zaslavsky \cite{Z2} applied this construction to group-labeled graphs (also called gain graphs), by showing that for each graph $G$ with edges labeled by a group, one can use the group-labeling to obtain a natural linear class $\cB$ of circuits of the graphic matroid $M(G)$.
The circuits in $\cB$ are the \emph{balanced} cycles of $G$ with respect to the group-labeling, and the pair $(G,\cB)$ is a \emph{biased graph}.
So, via Brylawski's construction, each graph $G$ with edges labeled by a group naturally leads to an elementary lift $M$ of $M(G)$; Zaslavsky calls this the \emph{lift matroid} of the biased graph $(G,\cB)$.
We will define this terminology in greater detail in Section 2.
Note that $M$ respects the group-labeling in the sense that a cycle of $G$ is a circuit of $M$ if and only if it is balanced.

The goal of this paper is to generalize the constructions of Brylawski and Zaslavsky to rank-$k$ lifts for $k\ge 2$.
A matroid $K$ is a \emph{lift} of a matroid $M$ if there is a matroid $P$ and a subset $X$ of $E(P)$ so that $P\del X=K$ and $P/X=M$.
If $r(K)-r(M)=k$, then $K$ is a \emph{rank-$k$ lift} of $M$.
We prove the following generalization of Brylawski's construction, by defining a lift of a matroid $M$ using a given matroid $N$ on the set of circuits of $M$.
Just as Brylawski's construction only applies to a linear class of circuits, the matroid $N$ must satisfy some necessary dependencies of the circuits of $M$.
We say that a collection $\cC'$ of circuits of $M$ is \emph{perfect} if $|\cup_{C\in\cC'} C|-r_M(\cup_{C\in\cC'} C)=|\cC'|$, and no circuit in $\cC'$ is contained in the union of the others.
The matroid $N$ must satisfy the natural condition that if a circuit $C$ of $M$ is contained in the union of the circuits of a perfect collection $\cC'$, then $C$ is spanned by $\cC'$ in $N$.

\begin{theorem}\label{main}
Let $M$ be a matroid, and let $N$ be a matroid on the set of circuits of $M$ so that 
\begin{enumerate}[$(\ref{main}*)$]
\item if $\cC'$ is a perfect collection of circuits of $M$, then each circuit $C'$ of $M$ contained in $\cup_{C\in\cC'}C$ satisfies $C'\in\cl_N(\cC')$.
\end{enumerate}
Then the function $r_{M^N}$ defined, for all $X \subseteq E(M)$, by 
$$r_{M^N}(X)=r_M(X)+r_N(\{C \colon \textrm{$C$ is a circuit of $M|X$}\})$$ 
is the rank function of a rank-$r(N)$ lift $M^N$ of $M$.
\end{theorem}

We remark that the case when $r(N) = 1$ is exactly Brylawski's construction.
To see this, note that if $C_1$ and $C_2$ are loops of $N$ for which $|C_1 \cup C_2| - r_M(C_1 \cup C_2) = 2$, then condition $(\ref{main}*)$ implies that every circuit $C$ of $M$ contained in $C_1 \cup C_2$ is also a loop of $N$.
This implies that the set of circuits of $M$ that are loops of $N$ is a linear class.
From the rank function of $M^N$, we see that the loops of $N$ are precisely the circuits of $M$ that are also circuits of $M^N$.
So when $r(N) = 1$, the matroid $M^N$ is equal to the matroid $M'$ obtained by applying Theorem \ref{brylawski} with the linear class of loops of $N$.

As an example of Theorem \ref{main} when $r(N) > 1$, let $M = U_{1,n}$.
Then the set of circuits of $M$ is the set of $2$-element subsets of $[n]$, which is the edge set of the complete graph $K_n$.
So we can think of the matroid $N = M(K_n)$ as a matroid on the set of circuits of $U_{1,n}$.
It turns out that $M(K_n)$ satisfies $(\ref{main}*)$.
To see this, let $\cC'$ be a perfect collection of circuits of $U_{1,n}$.
Let $H$ be the corresponding subgraph of $K_n$, so $E(H) = \cC'$ and $V(H) = \cup \cC'$.
Then each edge $\{i,j\}$ of $H$ is a leaf of $H$, or else the circuit $\{i,j\}$ of $U_{1,n}$ is contained in the union of the other circuits in $\cC'$.
In particular, $H$ has no cycles.
Also, $|\cup_{C\in\cC'}C| - 1 = |\cC'|$, because $\cC'$ is perfect and $r_M(\cup_{C\in\cC'}C) = 1$.
Since $\cup_{C\in\cC'}C = V(H)$ and $\cC' = E(H)$, this implies that $V(H) = E(H) + 1$. 
Then since $H$ is acyclic, this implies that $H$ is a tree.
Since $H$ is a tree for which every edge is a leaf, it follows that $H$ is a star.
If a circuit $C' = \{i,j\}$ of $U_{1,n}$ is contained in $\cup_{C\in\cC'}C$, then $i,j \in V(H)$, so the edge $\{i,j\}$ is spanned in $N$ by $E(H)$.
Equivalently, $C'$ is spanned in $N$ by $\cC'$.
Therefore, $N = M(K_n)$ satisfies $(\ref{main}*)$, so by Theorem \ref{main} we can construct the matroid $U_{1,n}^{M(K_n)}$.
This matroid has $n$ elements and rank $r(U_{1,n}) + r(M(K_n)) = 1 + (n - 1) = n$, and is thus the free matroid on ground set $[n]$.

We apply Theorem \ref{main} to group-labeled graphs, taking the matroid $M$ to be the cycle matroid of the underlying graph, and using the group-labeling to construct the matroid $N$, just as Zaslavsky uses the group-labeling to construct the balanced cycles.
We remark that, given a graph $G$ and a set $\cB$ of balanced cycles with respect to a group-labeling of $G$, we can define a rank-$1$ matroid $N$ on the circuits of $M(G)$, where each cycle in $\cB$ is a loop of $N$, and every other cycle of $G$ is a non-loop of $N$.
Then the matroid $M(G)^N$ is exactly Zaslavsky's lift matroid of the biased graph $(G, \cB)$.
For certain finite groups, we will construct a matroid $N$ of rank at least two on the circuits of $M(G)$ using the group-labeling.
For each finite group $\Gamma$ and integer $n\ge 3$, we only consider the \emph{fully $\Gamma$-labeled} graph $K_n^{\Gamma}$, which has vertex set $[n]$ and edge set ${[n]\choose 2}\times \Gamma$.
To ensure that the lift $M$ of $M(K_n^{\Gamma})$ respects the $\Gamma$-labeling, we require that the cycles of $K_n^{\Gamma}$ that are circuits of $M$ are precisely the balanced cycles of $K_n^{\Gamma}$.
If $\Gamma$ is isomorphic to $\bZ_p^j$ (the direct sum of $j$ copies of the cyclic group of order $p$) for some prime $p$ and some integer $j\ge 2$, then such a matroid $M$ exists.

\begin{theorem} \label{groups construction}
Let $p$ be a prime, and let $n\ge 3$ and $j\ge 2$ be integers.
For each integer $i$ with $1\le i\le j$, there is a rank-$i$ lift $M$ of $M(K_n^{\bZ_p^j})$ so that a cycle of $K_n^{\bZ_p^j}$ is a circuit of $M$ if and only if it is balanced.
\end{theorem}

Surprisingly, these are the only finite abelian groups for which this construction is possible.
If $\Gamma$ is a finite abelian group that is not isomorphic to $\bZ_p^j$ for some prime $p$ and integer $j\ge 2$, then there is no rank-$k$ lift of $M(K_n^{\Gamma})$ with $k\ge 2$ that respects the group-labeling.

\begin{theorem} \label{groups converse}
Let $\Gamma$ be a nontrivial finite abelian group, and let $n\ge 3$ be an integer.
Let $M$ be a lift of $M(K_n^{\Gamma})$ so that a cycle of $K_n^{\Gamma}$ is a circuit of $M$ if and only if it is balanced.
Then either $\Gamma\cong \bZ_p^j$ for some prime $p$ and integer $j\ge 2$, or $M$ is an elementary lift of $M(K_n^{\Gamma})$.
\end{theorem}

We conjecture that this result in fact holds for all finite groups, and we make some partial progress in support of this conjecture.


After applying Theorem \ref{main} to group-labeled graphs, we consider Theorem \ref{main} in its own right.
This theorem relies on the matroid $N$, and in general it is unclear how to construct a matroid $N$ of rank at least two on the set of circuits of $M$ so that $N$ satisfies $(\ref{main}*)$.
We show that, for any matroid $M$ with corank at least three, there is a rank-$3$ matroid $N$ on the circuits of $M$ that satisfies $(\ref{main}*)$, and there are many rank-$2$ matroids on the circuits of $M$ that satisfy $(\ref{main}*)$.
For representable matroids we can do much better.
Using the derived matroid of Longyear \cite{Longyear} and Oxley and Wang's recent generalization of it \cite{Oxley+Wang}, we show that if $M$ is representable, then there are many examples of a matroid $N$ on the circuits of $M$ that satisfies $(\ref{main}*)$.

\begin{theorem} \label{representable}
Let $M$ be a matroid representable over a field $\bF$.
Then, for each integer $k$ with $1 \le k \le r^*(M)$, there is a rank-$k$ matroid $N$ on the set of circuits of $M$ so that $N$ satisfies $(\ref{main}*)$.
\end{theorem}


It is unclear in general when there is a matroid $N$ on the circuits of a given matroid $M$ so that $N$ satisfies $(\ref{main}*)$.
Theorem \ref{brylawski} states that every elementary lift of a given matroid $M$ arises from a linear class of circuits of $M$.
Is it possible that every lift of $M$ arises from a matroid $N$ on the circuits of $M$ that satisfies $(\ref{main}*)$? We conjecture an affirmative answer to this question.

\begin{conjecture} \label{converse}
Let $M$ be a matroid. For every lift $K$ of $M$, there is a matroid $N$ on the circuits of $M$ so that $N$ satisfies $(\ref{main}*)$, and $K\cong M^N$. 
\end{conjecture}

We further conjecture how to construct $N$ given a lift $K$ of $M$ (Conjecture \ref{matroid}).


Finally, we consider Theorem \ref{main} in the dual setting.
If a matroid $K$ is an (elementary) lift of a matroid $M$, then $M$ is an \emph{(elementary) projection} or \emph{quotient} of $K$.
Since projections are dual to lifts, one can use Theorem \ref{brylawski} to construct an elementary projection of $M^*$ from a set of hyperplanes of $M^*$ whose complements form a linear class of circuits of $M$.
Such a collection of hyperplanes is called a \emph{linear subclass} of hyperplanes.
Equivalently, a collection $\cH$ of hyperplanes of a matroid $K$ is a linear subclass if, whenever $H_1$ and $H_2$ are hyperplanes in $\cH$ so that $r_K(H_1\cap H_2)=r(K)-2$, then every hyperplane $H$ of $K$ containing $H_1\cap H_2$ is also in $\cH$.
The following classical result of Crapo \cite{Crapo} constructs an elementary projection from a linear subclass; it is dual to Theorem \ref{brylawski}.

\begin{theorem}[Crapo '65] \label{crapo}
Let $K$ be a matroid, and let $\cH$ be a linear subclass of hyperplanes of $K$. 
Then the function $r_{K'}$ defined, for all $X \subseteq E(K)$, by 
$$r_{K'}(X)=\begin{cases}
r_K(X)-1 & \text{ if each hyperplane of $K$ containing $X$ is in $\cH$},  \\
r_K(X) & \text{ otherwise} 
\end{cases}$$
is the rank function of an elementary projection $K'$ of $K$.
Moreover, every elementary projection of $K$ can be obtained in this way.
\end{theorem}

In the same way, we can use Theorem \ref{main} to construct a projection of a matroid $K$ using a given matroid on the set of hyperplanes of $K$.


\begin{theorem} \label{dual main}
Let $K$ be a matroid, and let $N$ be a matroid on the set of hyperplanes of $K$ so that
\begin{enumerate}[$(\ref{dual main}*)$]
\item if $\cH'$ is a set of hyperplanes of $K$ for which no hyperplane  in $\cH'$ contains the intersection of the others, and $r_K(\cap_{H\in\cH'} H)=r(K)-|\cH'|$, then each hyperplane $H'$ of $K$ that contains $\cap_{H\in\cH'}H$ satisfies $H'\in \cl_N(\cH').$
\end{enumerate}
Then the function $r_{K_N}$ defined, for all $X \subseteq E(K)$, by
\begin{align*}
r_{K_N}(X) = r_M(X) & -r(N) \\
&+ r_N(\{H \colon \textrm{$H$ is a hyperplane of $K$ that contains $X$}\})
\end{align*}
is the rank function of a rank-$r(N)$ projection $K_N$ of $K$.
\end{theorem}

Condition $(\ref{dual main}*)$ implies that the set $\cH$ of hyperplanes of $K$ that are loops of $N$ is a linear subclass of hyperplanes, so when $r(N) = 1$, the matroid $K_N$ is equal to the matroid $K'$ obtained by applying Theorem \ref{crapo} to the linear subclass $\cH$.
By duality, Conjecture \ref{converse} is equivalent to the conjecture that every projection of $K$ arises from a matroid $N$ on the hyperplanes of $K$ that satisfies $(\ref{dual main}*)$. 
\end{section}


\begin{section}{Preliminaries}
Unless stated otherwise, we follow the notation and terminology of Oxley \cite{Oxley}.
Zaslavsky \cite{Z2} showed that elementary lifts of graphic matroids can be encoded using biased graphs.
A \emph{theta graph} consists of two distinct vertices $x$ and $y$, and three pairwise internally disjoint paths from $x$ to $y$.
A set $\cB$ of cycles of a graph $G$ satisfies the \emph{theta property} if no theta subgraph of $G$ contains exactly two cycles in $\cB$; equivalently, $\cB$ is a linear class of circuits of the graphic matroid $M(G)$.
A \emph{biased graph} is a pair $(G,\cB)$ where $\cB$ is a collection of cycles of $G$ which satisfies the theta property. 
The cycles in $\cB$ are \emph{balanced}, and the cycles not in $\cB$ are \emph{unbalanced}.

Biased graphs were first described by Zaslavsky in \cite{Z1}, and in \cite{Z2} he defined the following matroid associated with a given biased graph.
The \emph{lift matroid} of a biased graph $(G,\cB)$ is the matroid with ground set $E(G)$ so that $I\subseteq E(G)$ is independent if and only if the subgraph of $G$ induced by $I$ contains at most one cycle, and no balanced cycle.
Equivalently, the lift matroid of $(G, \cB)$ is the matroid obtained from applying Theorem \ref{brylawski} to $M(G)$ with the linear class $\cB$.
Note that if $\cB$ is the set of all cycles of $G$, then the lift matroid of $(G,\cB)$ is isomorphic to $M(G)$.

A natural family of biased graphs arises from graphs whose edges are labeled by elements of a group.
We define group-labeled graphs using the notation of \cite{DeVos}.
A \emph{group-labeling} of a graph $G$ consists of an orientation of the edges of $G$, and a function $\phi\colon E(G)\to \Gamma$ for some (multiplicative) group $\Gamma$.
For each walk $W$ in $G$ with edge sequence $e_1,e_2,\dots,e_k$, define $c_i(W)$ by
$$c_i(W)=\begin{cases}
1 & \text{ if $e_i$ is traversed forward in $W$},  \\
-1 & \text{ if $e_i$ is traversed backward in $W$}, 
\end{cases}$$
and define $\phi(W)=\prod_{i=1}^{k}\phi(e_i)^{c_i(W)}$.
Let $\cB_{\phi}$ be the set of cycles $C$ of $G$ for which some (and thus every) simple closed walk $W$ around $C$ satisfies $\phi(W) = 1$.
Then $\cB_{\phi}$ satisfies the theta property, so $(G,\cB_{\phi})$ is a biased graph \cite{Z1}.
The cycles in $\cB_{\phi}$ are the balanced cycles with respect to the group-labeling.
We will also be interested in the group values of simple closed walks around unbalanced cycles of $(G, \cB_{\phi})$.
For each cycle $C$ of $G$, we define
 $$\phi(C)=\{\phi(W)\colon W \textrm{ is a simple closed walk around } C\}.$$
Note that $C$ is balanced if and only if $1 \in \phi(C)$.

For each finite group $\Gamma$ and each integer $n\ge 3$, we write $K_n^{\Gamma}$ for the graph with vertex set $[n]$ and edge set ${[n]\choose 2}\times \Gamma$.
We write $\cB_n^{\Gamma}$ for the set of balanced cycles obtained from the $\Gamma$-labeling $\phi((\{i,j\}, \alpha)) = \alpha$, and the following edge-orientation: for all $1\le i<j\le n$, each edge between vertices $i$ and $j$ is oriented from vertex $i$ to vertex $j$.
We say that a cycle of $K_n^{\Gamma}$ is balanced if it is in $\cB_n^{\Gamma}$; the natural $\Gamma$-labeling and edge-orientation will be implicit throughout the remainder of the paper.
We denote the lift matroid of $(K_n^{\Gamma},\cB_n^{\Gamma})$ by $\LG(n,\Gamma)$.
\end{section}


\begin{section}{The Construction}
In this section we prove Theorem \ref{main}.
For a set $E$ and a set $\cX$ of subsets of $E$, we write $\cup\cX$ for $\cup_{X\in\cX}X$.
Given a matroid $M$, we write $\cC(M)$ for the collection of circuits of $M$.
Recall that a collection $\cC'$ of circuits of a matroid $M$ is \emph{perfect} if $|\cup\cC'|-r_M(\cup\cC')=|\cC'|$, and no circuit in $\cC'$ is contained in the union of the others.
The following lemma shows that the set of fundamental circuits of a matroid with respect to some basis is always a perfect collection of circuits.

\begin{lemma}\label{find perfect}
Let $M$ be a matroid.
Let $B$ be a basis of $M$, and, for each $e\in E(M) - B$, let $C_e$ be the unique circuit of $M|(B\cup\{e\})$.
Then $\cC'=\{C_e\colon e\in E(M) - B\}$ is a perfect collection of $|E(M)|-r(M)$ circuits of $M$.
Conversely, every perfect collection of $|E(M)|-r(M)$ circuits of $M$ arises in this way.
\end{lemma}
\begin{proof}
Let $E = E(M)$.
Clearly no circuit of $\cC'$ is contained in the union of the others, and $\cup \cC' = E$.
Then $|\cup \cC'| - r_M(\cup \cC') = |E| - r(M) = |\cC'|$, so $\cC'$ is perfect.
Conversely, let $\cC'$ be a perfect collection of $|E(M)|-r(M)$ circuits of $M$.
For each $C \in \cC'$, let $X_C$ be the set of elements in $C$ that are not in any other set in $\cC'$.
Let $T$ be a transversal of $\{X_C \colon C \in \cC'\}$.
Then $|T| = |E| - r(M)$, and $T \subseteq \cl_M((\cup \cC') - T)$.
This implies that $E - T$ is a spanning subset of $M$ of size $r(M)$, and is thus a basis of $M$.
Also, each set in $\cC'$ has all but one element in $E - T$, so $\cC'$ is the set of fundamental circuits of $M$ with respect to the basis $E - T$.
\end{proof}

The following theorem easily implies Theorem \ref{main}.

\begin{theorem} \label{construction}
Let $M$ be a matroid, and let $N$ be a matroid on the set of circuits of $M$ so that 
\begin{enumerate}[$(\ref{construction}*)$]
\item if $\cC'$ is a perfect collection of circuits of $M$, then each circuit $C'$ of $M$ contained in $\cup \cC'$ satisfies $C'\in\cl_N(\cC')$.
\end{enumerate}
Let $\cI$ denote the collection of sets $X\subseteq E(M)$ for which there exists a collection $\cC'$ of $|X|-r_M(X)$ circuits of $M|X$ so that $\cC'$ is independent in $N$.
Then $\cI$ is the collection of independent sets of a matroid $M^N$ which is a rank-$r(N)$ lift of $M$.
Moreover, the function $r_{M^N}$ defined, for all $X \subseteq E(M)$, by 
$$r_{M^N}(X)=r_M(X)+r_N(\{C \colon \textrm{$C$ is a circuit of $M|X$}\})$$ 
is the rank function of $M^N$.
\end{theorem}
\begin{proof}
We first use $(\ref{construction}*)$ to relate subsets of $E(M)$ with perfect collections of circuits.

\begin{claim} \label{perfects}
Let $X\subseteq E(M)$, and let $\cC'$ be a perfect collection of $|X|-r_M(X)$ circuits of $M|X$.
Then $\cC(M|X)\subseteq \cl_N(\cC')$.
Moreover, if $X\in \cI$, then $\cC'$ is independent in $N$.
\end{claim}
\begin{proof}
We have $|\cup\cC'|-r_M(\cup\cC')=|\cC'|=|X|-r_M(X)$, where the first equality holds because $\cC'$ is perfect.
If there is a circuit $C$ of $M|X$ that is not contained in $\cup \cC'$, then $\cC' \cup \{C\}$ is a collection of greater than $|X| - r_M(X)$ circuits of $M|X$ such that none is contained in the union of the others, a contradiction.
Thus, each circuit of $M|X$ is contained in $\cup\cC'$.
Then $(\ref{construction}*)$ implies that $\cC(M|X)\subseteq \cl_N(\cC')$.
In particular, this implies that $r_N(\cC(M|X)) \le r_N(\cC')$.
Now suppose that $X \in \cI$.
By the definition of $\cI$, we know that $r_N(\cC(M|X))\ge |X|-r_M(X)$.
Then $r_N(\cC') \ge |X|-r_M(X)$, and since $|X| - r_M(X) = |\cC'|$, this implies that $r_N(\cC') \ge |\cC'|$, so $\cC'$ is independent in $N$.
\end{proof}

We now show that $\cI$ is the collection of independent sets of a matroid on $E(M)$.
Clearly $\varnothing\in \cI$, since $\varnothing$ is independent in $N$. 
Let $Y$ be a nonempty set in $\cI$, let $e\in Y$, and let $X=Y-\{e\}$.
We will show that $X\in\cI$.
If $e\notin\cl_M(X)$, then $|X|-r_M(X)=|Y|-r_M(Y)$, and $M|X$ and $M|Y$ have the same circuits.
Since $Y\in\cI$, this implies that $X\in\cI$, so we may assume that $e\in \cl_M(X)$.
Let $B$ be a basis of $M|X$, and, for each $e'\in Y-B$, let $C_{e'}$ denote the unique circuit of $M|(B\cup\{e'\})$.
Let $\cC'=\{C_{e'}\colon e'\in Y-B\}$.
Then, by Lemma \ref{find perfect} applied to $M|Y$, $\cC'$ is a perfect collection of $|Y|-r_M(Y)$ circuits of $M|Y$, and, by \ref{perfects}, $\cC'$ is independent in $N$.
Then $\cC'-\{C_e\}$ is a collection of $|X|-r_M(X)$ circuits of $M|X$ which is independent in $N$, so $X\in \cI$.


We now show that $\cI$ satisfies the augmentation property.

\begin{claim} \label{axiom2}
Let $I_1$ and $I_2$ be sets in $\cI$ so that $|I_1|<|I_2|$. Then there is some $e\in I_2-I_1$ so that $I_1\cup\{e\}\in\cI$.
\end{claim}
\begin{proof}
Suppose that there is no element $e\in I_2-I_1$ so that $I_1\cup\{e\}\in\cI$.
We first show that $I_2\subseteq \cl_M(I_1)$.
If not, then let $e\in I_2-\cl_M(I_1)$.
Since $|I_1\cup\{e\}|-r_M(I_1\cup\{e\})=|I_1|-r_M(I_1)$, and $\cC(M|(I_1\cup\{e\}))=\cC(M|I_1)$, we have $I_1\cup\{e\}\in\cI$, a contradiction.
Thus, $I_2\subseteq \cl_M(I_1)$, and so $|I_2|-r_M(I_2)> |I_1|-r_M(I_1)$.

Let $B$ be a basis of $M|I_1$. 
For each $e\in (I_1\cup I_2)-B$, let $C_e$ denote the unique circuit of $M|(B\cup\{e\})$.
Let $\cC_1=\{C_e\colon e\in I_1-B\}$, and let $\cC_2=\{C_e\colon e\in I_2-I_1\}$.
Then, by Lemma \ref{find perfect} applied to $M|I_1$ and $M|(I_1\cup I_2)$, $\cC_1$ is a perfect collection of $|I_1|-r_M(I_1)$ circuits of $M|I_1$, and $\cC_1\cup \cC_2$ is a perfect collection of $|I_1\cup I_2|-r_M(I_1\cup I_2)$ circuits of $M|(I_1\cup I_2)$.
Then \ref{perfects} with $(X,\cC')=(I_1,\cC_1)$ implies that $\cC(M|I_1)\subseteq \cl_N(\cC_1)$, and that $\cC_1$ is independent in $N$.
Also, \ref{perfects} with $(X,\cC')=(I_1\cup I_2,\cC_1\cup\cC_2)$ implies that $\cC(M|I_2)\subseteq \cl_N(\cC_1\cup\cC_2)$.

Since $\cC_1$ is independent in $N$, we have $r_N(\cC_1\cup\cC_2)=r_N(\cC_1)$, or else there is some $e\in I_2-I_1$ so that $I_1\cup\{e\}\in\cI$, by the definitions of $\cI$ and $\cC_2$.
But then $$r_N(\cC(M|I_2))\le r_N(\cC_1\cup\cC_2) \le |\cC_1|= |I_1|-r_M(I_1)<|I_2|-r_M(I_2),$$
which contradicts that $I_2\in\cI$.
\end{proof}

We now know that $\cI$ is the collection of independent sets of a matroid $M^N$ on $E(M)$.
Next we compute the rank function of $M^N$.
Let $X\subseteq E(M)$, and let $B$ be a basis of $M|X$.
For each $e\in X-B$, let $C_e$ denote the unique circuit of $M|(B\cup\{e\})$.
Let $\cC'=\{C_e\colon e\in X-B\}$.
Then by Lemma \ref{find perfect} applied to $M|X$, $\cC'$ is a perfect collection of $|X|-r_M(X)$ circuits of $M|X$, so \ref{perfects} implies that $r_N(\cC')=r_N(\cC(M|X))$.
Let $\cC''$ be a subset of $\cC'$ so that $|\cC''|=r_N(\cC')=r_N(\cC(M|X))$.
Then $|B\cup (\cup\cC'')|-r_M(B\cup (\cup\cC''))=|\cC''|$, by the definition of $\cC'$.
Since $\cC''$ is independent in $N$, this implies that $B\cup (\cup\cC'')$ is independent in $M^N$, by the definition of $\cI$.
Therefore, $$r_{M^N}(X)\ge |B\cup (\cup\cC'')|=|B|+|\cC''|=r_M(X)+r_N(\cC(M|X)).$$
We now show that the reverse inequality holds.
Let $I$ be a basis of $M^N|X$, and let $\cC'$ be a collection of $|I|-r_M(I)$ circuits of $M|I$ so that $\cC'$ is independent in $N$.
Then $|I|-r_M(I)\le r_N(\cC(M|I))$, and so 
$$r_{M^N}(X)=|I|\le r_M(I)+r_N(\cC(M|I))\le r_M(X)+r_N(\cC(M|X)),$$
since $I\subseteq X$.
From the rank function of $M^N$, it is easy to see that $r(M^N)= r(M)+r(N)$.


Finally, we show that each circuit of $M^N$ is a union of circuits of $M$; this implies that $M^N$ is a lift of $M$ (see \cite[Prop. 7.3.6]{Oxley}).
Let $Y$ be a circuit of $M^N$.
If $Y$ has an element $e$ not in any circuit of $M|Y$, then from the rank function of $M^N$ it follows that $r_{M^N}(Y-\{e\})<r_{M^N}(Y)$, since $\cC(M|(Y-\{e\}))=\cC(M|Y)$.
But then $Y$ is not a circuit of $M^N$.
Thus, each element of $Y$ is in a circuit of $M|Y$, so $Y$ is a union of circuits of $M$.
\end{proof}

\end{section}


\begin{section}{The Construction for $\bZ_p^j$-Labeled Graphs}
In this section we prove Theorem \ref{groups construction}.
Let $j \ge 2$ be an integer, and let $p$ be a prime.
Recall that each cycle $C$ of $K_n^{\bZ_p^j}$ has a set $\phi(C)\subseteq \bZ_p^j$ of values of simple closed walks around $C$, and that $C$ is balanced if and only if $\phi(C)$ contains the identity element of $\bZ_p^j$.
It is not hard to see that the set $\phi(C)$ is closed under inverses, since taking the reverse of a simple closed walk results in the inverse group element in $\bZ_p^j$. 
Also, since $\bZ_p^j$ is abelian, any two simple closed walks around $C$ with the same cyclic ordering have the same value in $\bZ_p^j$, so $|\phi(C)|\le 2$.

For each proper divisor $i$ of $j$, the group $(\bZ_p^i)^{j/i}$ is isomorphic to $\bZ_p^j$, and is the additive group of the vector space $\GF(p^i)^{j/i}$.
Thus, there is a natural map $f_i$ from $\bZ_p^j$ to the vector space $\GF(p^i)^{j/i}$.
This can be extended to a natural map $f_i'$ from $\bZ_p^j$ to the ground set of the projective geometry $\PG((j/i) - 1, p^i)$. 

For each unbalanced cycle $C$ of $K_n^{\bZ_p^j}$, the set $\phi(C)$ has size two and is closed under inverses; thus, both elements of $\bZ_p^j$ in $\phi(C)$ map to the same element of $\PG((j/i) - 1, p^i)$ under $f_i'$.
Therefore, the map $g_i$ from the set of unbalanced cycles of $K_n^{\bZ_p^j}$ to $E(\PG((j/i) - 1, p^i))$ defined by $g_i(C) = f'_i(\alpha)$ for some $\alpha \in \phi(C)$ is well-defined.
Since each projection (or lift) of $\PG((j/i) - 1, p^i)$ has the same ground set as $\PG((j/i) - 1, p^i)$, we can use $g_i$ to define a matroid on the set of cycles of $K_n^{\bZ_p^j}$ using any projection of $\PG((j/i) - 1, p^i)$.
The following definition makes this idea precise.

\begin{definition} \label{cycle map}
Let $n \ge 3$ and $j\ge 2$ be integers, and let $p$ be a prime. Let $i$ be a positive divisor of $j$, and let $K$ be a projection of $\PG((j/i) - 1, p^i)$.
We define $N=N(n,j,p,K)$ to be the matroid on the set of cycles of $K_n^{\bZ_p^j}$ for which each set $\cC$ of unbalanced cycles satisfies
$$r_N(\cC)=r_K(\{g_i(C)\colon C\in \cC\}),$$
and each balanced cycle is a loop of $N$.
\end{definition}

If $K = \PG((j/i) - 1, p^i)$, then $r_N(\cC)$ is simply the rank in the vector space $\GF(p^i)^{j/i}$ of the set of vectors associated with the cycles in $\cC$ by the map $f_i$.
Also, if $K$ is a projection of $\PG((j/i) - 1, p^i)$, then the map $g_i$ shows that $N(n,j,p,K)$ is a projection of $N(n,j,p,\PG((j/i) - 1, p^i))$.

In order to apply Theorem \ref{main} with $M=M(K_n^{\bZ_p^j})$ and $N=N(n,j,p,K)$, we must show that $N$ satisfies $(\ref{main}*)$.
The following lemma shows that we need only consider the case in which $K$ is actually a projective geometry.

\begin{proposition} \label{project N}
Let $M$ be a matroid. If $N$ is a matroid on $\cC(M)$ that satisfies $(\ref{main}*)$, then every projection of $N$ also satisfies $(\ref{main}*)$.
\end{proposition}
\begin{proof}
Let $N'$ be a projection of $N$.
Let $\cC'$ be a perfect collection of circuits of $M$, and let $C$ be a circuit of $M$ so that $C\subseteq \cup\cC'$.
Then $C\in\cl_N(\cC')$, by $(\ref{main}*)$. 
Since $N'$ is a projection of $N$, we have $\cl_N(\cC')\subseteq \cl_{N'}(\cC')$ (see \cite[Prop. 7.3.6]{Oxley}).
Thus, $C\in \cl_{N'}(\cC')$, as desired.
\end{proof}

We now show that the matroid $N=N(n,j,p,K)$ on the circuits of $M(K_n^{\bZ_p^j})$ satisfies $(\ref{main}*)$.

\begin{proposition} \label{perfect}
Let $n \ge 3$ and $j\ge 2$ be integers, and let $p$ be a prime.
Let $i$ be a positive divisor of $j$, and let $N=N(n,j,p,\PG((j/i) - 1, p^i))$.
If $\cC'$ is a perfect collection of circuits of $M(K_n^{\bZ_p^j})$ and $C$ is a circuit of $M(K_n^{\bZ_p^j})$ contained in $\cup\cC'$, then $C\in \cl_{N}(\cC')$.
\end{proposition}
\begin{proof}
We write $G=K_n^{\bZ_p^j}$, for convenience.
Let $|\cC'|$ be minimal so that the claim is false; then $|\cC'|\ge 2$.
Let $C$ be a circuit of $M(G)$ contained in $\cup\cC'$.
We freely use the fact that each subset of $\cC'$ is also a perfect collection of circuits of $M(G)$.
This implies that for each cycle $C'\in \cC'$, there is at least one edge in $C'\cap C$ that is not in any other cycle in $\cC'$.
Let $X$ be a transversal of these edges of $\cup \cC'$; note that $X\subseteq C$.
Then $$|(\cup\cC')-X| = |\cup \cC'| - |X| = r_{M(G)}(\cup \cC') + |\cC'| - |X| = r_{M(G)}(\cup \cC'),$$ since $\cC'$ is perfect, so 
$T=(\cup\cC')-X$ is the edge-set of a spanning forest of the graph $G[\cup\cC']$.
By the minimality of $|\cC'|$ and the fact that $G[C]$ is connected, this forest is in fact a tree.
We use the tree $T$ to define two cycles.

\begin{claim} \label{get theta}
There are cycles $C_1,C_2$ contained in $\cup\cC'$ which form a theta graph with $C$, such that $C_1,C_2\in \cl_N(\cC')$.
\end{claim}
\begin{proof}
Let $v_1,v_2$ be distinct vertices of $G[C]$, and let $P$ be the unique path in $T$ from $v_1$ to $v_2$.
Let $v_1'$ be the first vertex of $P$ in $C$ other than $v_1$, and let $P'$ be the segment of $P$ from $v_1$ to $v_1'$.
Then $P'$ is a path with both ends in $C$ which is internally vertex-disjoint from $C$, so $G[C\cup P']$ is a theta graph.

Let $P_1,P',P_2$ denote the three internally vertex-disjoint paths of the theta graph $G[C\cup P']$ from $v_1$ to $v_1'$.
Then $P_1$ and $P_2$ each contain an element of $X$;  if $P_i$ and $X$ are disjoint, then $G[P_{3-i}\cup P']$ is a circuit of the tree $T$.
Let $e_1\in P_1\cap X$ and let $e_2\in P_2\cap X$, and for each $i\in \{1,2\}$ let $C_i'$ denote the circuit in $\cC'$ that contains $e_i$.

For each $i\in \{1,2\}$ let $C_i=P_i\cup P'$; then $C_i$ is a cycle that does not contain $e_{3-i}$.
Since $\cC'$ is perfect, no element of $C_2' - (\cup\cC')$ is in a cycle of $G[(\cup\cC') - \{e_2\}]$; otherwise $|\cup \cC'| - r_{M(G)}(\cup \cC') > |\cC'|$.
Since $C_1$ does not contain $e_2$ and $C_1$ is a cycle, it follows that $C_1 \subseteq \cup (\cC'-\{C_2'\})$.
Similarly, $C_2 \subseteq \cup (\cC'-\{C_1'\})$.
So by the minimality of $|\cC'|$, we have $C_2\in \cl_N(\cC'-\{C_1'\})$ and $C_1\in \cl_N(\cC'-\{C_2'\})$, and thus $C_1,C_2\in \cl_N(\cC')$.
\end{proof}

We now use the definition of $N$.

\begin{claim} \label{use theta}
$C\in \cl_N(\{C_1,C_2\})$.
\end{claim}
\begin{proof}
Assume without loss of generality that the edges of $P'$ and $P_1$ are oriented from $v_1$ to $v_1'$, and the edges of $P_2$ are oriented from $v_1'$ to $v_1$.
Let $\alpha_1$ and $\alpha'$ denote the values in $\bZ_p^j$ of the walks from $v_1$ to $v_1'$ on $P_1$ and $P'$, respectively, and let $\alpha_2$ denote the value of the walk from $v_1'$ to $v_1$ on $P_2$.
Then $\alpha_1+\alpha_2\in \phi(C)$, while $\alpha_1-\alpha'\in \phi(C_1)$ and $\alpha'+\alpha_2\in \phi(C_2)$.
By the definition of $N=N(n,j,p,\PG((j/i) - 1, p^i))$, this implies that $\{C, C_1, C_2\}$ is a circuit of $N$, and so the claim holds. 
\end{proof}

Claims \ref{get theta} and \ref{use theta} combine to show that $C\in \cl_N(\{\cC'\})$.
\end{proof}


By Proposition \ref{perfect}, we may use Theorem \ref{main} to define the matroid $M(K_n^{\bZ_p^j})^{N(n,j,p,K)}$ for each prime $p$, each integer $j\ge 2$, each positive divisor $i$ of $j$, each integer $n\ge 3$, and each projection $K$ of $\PG((j/i) - 1, p^i)$.
This matroid is a rank-$r(K)$ lift of $M(K_n^{\bZ_p^j})$.
We now prove Theorem \ref{groups construction}. 

\begin{proof}[Proof of Theorem \ref{groups construction}]
By Proposition \ref{perfect}, the matroid $N(n,j,p,\PG(j-1, p))$ on the set of cycles of $K_n^{\bZ_p^j}$ satisfies $(\ref{main}*)$.
Also, by Definition \ref{cycle map}, a cycle of $K_n^{\bZ_p^j}$ is a loop of $N(n,j,p,\PG(j-1, p))$ if and only if it is balanced.
Let $K$ be the $(j-i)$-th truncation of $\PG(j-1, p)$.
Then $K$ is a projection of $\PG(j - 1,p)$, and each $X \subseteq E(\PG(j-1,p))$ satisfies $r_K(X) = \min(r_{\PG(j-1,p)}(X), i)$.
Since $K$ is a projection of $\PG(j-1, p)$, the map $g_1$ shows that $N(n,j,p,K)$ is a projection of $N(n,j,p,\PG(j-1,p))$.
Then since $N(n,j,p,\PG(j-1, p))$ satisfies $(\ref{main}*)$, Proposition \ref{project N} with $N = N(n,j,p,\PG(j-1, p))$ implies that $N(n,j,p,K)$ also satisfies $(\ref{main}*)$.
Also, each element $e$ of $\PG(j - 1,p)$ satisfies $r_K(\{e\}) = \min(r_{\PG(j-1,p)}(\{e\}), i) = 1$, so $K$ is loopless.
Since $K$ is loopless, by Definition \ref{cycle map}, a cycle of $K_n^{\bZ_p^j}$ is a loop of $N(n,j,p,K)$ if and only if it is balanced.
By Theorem \ref{main}, this implies that a cycle of $K_n^{\bZ_p^j}$ is a circuit of $M=M(K_n^{\bZ_p^j})^{N(n,j,p,K)}$ if and only if it is balanced.
Since $r(K)=i$, the matroid $M(K_n^{\bZ_p^j})^{N(n,j,p,K)}$ is a rank-$i$ lift of $M(K_n^{\bZ_p^j})$.
Thus, the theorem holds with $M=M(K_n^{\bZ_p^j})^{N(n,j,p,K)}$.
\end{proof}

To prove Theorem \ref{groups construction}, we used the construction of Theorem \ref{main}.
Conversely, we conjecture that every lift of $M(K_n^{\bZ_p^j})$ that has each balanced cycle of $K_n^{\bZ_p^j}$ as a circuit arises from this construction.

\begin{conjecture}
Let $n \ge 3$ and $j\ge 2$ be integers, let $p$ be a prime, and let $M$ be a lift of $M(K_n^{\bZ_p^j})$ so that each balanced cycle of $K_n^{\bZ_p^j}$ is a circuit of $M$.
Then there is a positive divisor $i$ of $j$ and a projection $K$ of $\PG((j/i) - 1, p^i)$ so that $M\cong M(K_n^{\bZ_p^j})^{N(n,j,p,K)}$.
\end{conjecture}

This may be easier to prove in the case that no unbalanced cycle of $K_n^{\bZ_p^j}$ is a circuit of $M$. 

\end{section}


\begin{section}{Other Abelian Groups}
In this section we prove Theorem \ref{groups converse}.
For each finite group $\Gamma$ and integer $n\ge 3$, we define $\cM_{n, \Gamma}$ to be the class of lifts $M$ of $M(K_n^{\Gamma})$ for which a cycle of $K_n^{\Gamma}$ is a circuit of $M$ if and only if it is a balanced cycle of $K_n^{\Gamma}$.
These are the lifts of $M(K_n^{\Gamma})$ that respect the $\Gamma$-labeling.
Note that each matroid in $\cM_{n, \Gamma}$ is simple, since each $2$-element cycle of $M(K_n^{\Gamma})$ is unbalanced.
Also, if $\Gamma$ is the trivial group, then $\cM_{n, \Gamma}$ is empty, so we will restrict our attention to nontrivial groups.

For each nontrivial finite group $\Gamma$ and integer $n\ge 3$, the class $\cM_{n, \Gamma}$ certainly contains the rank-$1$ lift $\LG(n,\Gamma)$.
Theorem \ref{groups converse} says that if $\Gamma$ is a nontrivial finite abelian group that is not isomorphic to $\bZ_p^j$ for some prime $p$ and integer $j\ge 2$, then $\cM_{n,\Gamma}$ contains only $\LG(n,\Gamma)$, up to isomorphism.
To prove this, we will use three lemmas, which each apply to arbitrary finite groups.

The first lemma uses local information about the balanced cycles of $K_n^{\Gamma}$.
For each element $\alpha\in \Gamma$, we write $E_{\alpha}$ for $\{(\{i,j\},\alpha)\colon 1\le i<j\le n\}$; these are the edges of $K_n^{\Gamma}$ labeled by $\alpha$.
More generally, for each set $A\subseteq \Gamma$, we write $E_A$ for $\{(\{i,j\},\alpha)\colon 1\le i<j\le n \textrm{ and } \alpha\in A\}$.
For convenience, for each $\alpha \in \Gamma$ and $1\le i<j\le n$, we write $\alpha_{ij}$ for the edge $(\{i,j\}, \alpha)$.


\begin{lemma} \label{alpha}
Let $n\ge 3$ be an integer, let $\Gamma$ be a finite group with identity $\ep$, and let $M\in \cM_{n,\Gamma}$.
Then each non-identity element $\alpha\in\Gamma$ satisfies $r_M(E_{\{\alpha, \ep\}})=n$, and $E_{\alpha}\cap \cl_M(E_{\ep})=\varnothing$.
\end{lemma}
\begin{proof}
Clearly $E_{\ep}$ spans $M(K_n^{\Gamma})$, since each element of $M(K_n^{\Gamma})$ is parallel to an element in $E_{\ep}$.
Let $B=\{\alpha_{12}\}\cup E_{\ep}$.
For each $3\le j\le n$, the cycle $\{\alpha_{12},\ep_{2j},\alpha_{1j}\}$ is balanced, and is thus a circuit of $M$, since $M\in \cM_{n,\Gamma}$.
Since $\{\alpha_{12},\ep_{2j},\alpha_{1j}\}$ is a circuit of $M$ for all $3\le j\le n$, it follows that $\alpha_{1j}\in\cl_M(B)$ for each $2\le j\le n$.
For all $2\le i<j\le n$, the cycle $\{\ep_{1i},\alpha_{ij},\alpha_{1j}\}$ is balanced, and is thus a circuit of $M$; this implies that $\alpha_{ij}\in\cl_M(B)$.
Thus, $E_{\alpha}\subseteq \cl_M(B)$, so $r_M(E_{\{\alpha, \ep\}}) \le r(M(K_n^{\Gamma}))+1 =n$.

We now show that $E_{\alpha}\cap \cl_M(E_{\ep})=\varnothing$; this implies that $r_M(E_{\{\alpha, \ep\}})=n$.
If $\alpha_{12}\in \cl_M(E_{\ep})$, then by the previous paragraph we have $E_{\alpha}\subseteq \cl_M(E_{\ep})$.
Since $M$ is a lift of $M(K_n^{\Gamma})$ and $M|E_{\ep}=M(K_n^{\Gamma})|E_{\ep}$, this implies that $M|(E_{\{\alpha, \ep\}})=M(K_n^{\Gamma})|(E_{\{\alpha, \ep\}})$.
But $M$ is simple and $M(K_n^{\Gamma})|(E_{\{\alpha, \ep\}})$ is not, so this is a contradiction.
Thus, $\alpha_{12}\notin \cl_M(E_{\ep})$.
The same argument applies to each element of $E_{\alpha}$, so $E_{\alpha}\cap \cl_M(E_{\ep})=\varnothing$.
\end{proof}


The next lemma uses a more global argument.

\begin{lemma} \label{generate}
Let $n\ge 3$ be an integer, let $\Gamma$ be a finite group with identity $\ep$, and let $M\in \cM_{n,\Gamma}$.
Let $A$ be a subset of $\Gamma$, and let $\langle A \rangle$ be the subgroup of $\Gamma$ generated by $A$.
Then $E_{\langle A \rangle} \subseteq \cl_M(E_{A\cup\{\ep\}})$.
\end{lemma}
\begin{proof}
Let $\circ$ be the (multiplicative) group operation of $\Gamma$.
We write $B=E_{A \cup \{\ep\}}$ for convenience.
Let $\alpha,\beta\in \Gamma$ so that $E_{\alpha}\cup E_{\beta}\subseteq \cl_M(B)$.
We will show that $E_{\alpha^{-1}}\subseteq \cl_M(B)$, and that $E_{\alpha\circ \beta}\subseteq \cl_M(B)$; then each element $\gamma\in \Gamma$ generated by $A$ satisfies $E_{\gamma}\subseteq \cl_M(B)$, and so $E_{\langle A \rangle}\subseteq \cl_M(B)$.
We freely use the fact that each balanced cycle of $K_n^{\Gamma}$ is a circuit of $M$, since $M\in \cM_{n,\Gamma}$.

Since $\{\alpha^{-1}_{12},\alpha_{23},\ep_{13}\}$ is a circuit of $M$ and $\alpha_{23},\ep_{13}\in \cl_M(B)$, we have $\alpha^{-1}_{12}\in \cl_M(B)$.
Then, since $\{\alpha_{12}^{-1},\ep_{2j},\alpha_{1j}^{-1}\}$ is a circuit of $M$ and $\alpha_{12}^{-1},\ep_{2j}\in \cl_M(B)$, we have $\alpha_{1j}^{-1}\in \cl_M(B)$ for each $2\le j\le n$.
Finally, since $\{\alpha_{1i},\alpha^{-1}_{ij},\ep_{1j}\}$ is a circuit of $M$ and $\alpha_{1i},\ep_{1j}\in \cl_M(B)$, we have $\alpha^{-1}_{ij}\in\cl_M(B)$ for all $2\le i<j\le n$, and thus $E_{\alpha^{-1}}\subseteq \cl_M(B)$.

We now show that $E_{\alpha\circ \beta}\subseteq \cl_M(B)$.
Since $\{(\alpha\circ\beta)_{12}, \beta^{-1}_{23}, \alpha_{13}\}$ is a circuit of $M$ and $\beta^{-1}_{23},\alpha_{13}\in \cl_M(B)$, we have $(\alpha\circ\beta)_{12}\in \cl_M(B)$.
Since $\{\beta^{-1}_{12}, \alpha^{-1}_{2j}, (\alpha\circ\beta)_{1j}\}$ is a circuit of $M$ and $\beta^{-1}_{12},\alpha^{-1}_{2j}\in\cl_M(B)$, we have $(\alpha\circ\beta)_{1j}\in \cl_M(B)$ for each $2\le j \le n$.
Finally, since $\{\alpha^{-1}_{1i}, (\alpha\circ\beta)_{ij}, \beta_{1j}\}$ is a circuit of $M$ and $\alpha^{-1}_{1i}, \beta_{1j}\in \cl_M(B)$, we have $(\alpha\circ\beta)_{ij}\in \cl_M(B)$ for all $2\le i<j\le n$. 
Thus $E_{\alpha\circ\beta}\subseteq \cl_M(B)$.
\end{proof}


The following lemma defines an equivalence relation on the non-identity elements of $\Gamma$.
Its proof follows without difficulty from Lemmas \ref{alpha} and \ref{generate}.

\begin{lemma} \label{equivalence}
Let $n\ge 3$ be an integer, let $\Gamma$ be a finite group with identity $\ep$, and let $M\in \cM_{n,\Gamma}$.
For $\alpha,\beta\in\Gamma-\{\ep\}$, write $\alpha\sim\beta$ if $r_M(E_{\{\alpha,\beta,\ep\}})=n$.
Then
\begin{enumerate}[(i)]
\item $\sim$ is an equivalence relation,

\item each equivalence class $A$ of $\sim$ satisfies $r_M(E_{A\cup \{\ep\}})=n$, and 

\item for each equivalence class $A$ of $\sim$, the set $A\cup\{\ep\}$ is a subgroup of $\Gamma$.
\end{enumerate}
\end{lemma}


We now prove the following restatement of Theorem \ref{groups converse}.

\begin{theorem} \label{p-group}
Let $n\ge 3$ be an integer, let $\Gamma$ be a finite abelian group, and let $M \in \cM_{n,\Gamma}$.
If $r(M) - r(M(K_n^{\Gamma})) > 1$, then there is a prime $p$ and an integer $j\ge 2$ so that $\Gamma\cong \bZ_p^j$.
\end{theorem} 
\begin{proof}
Let $\circ$ be the (multiplicative) group operation of $\Gamma$, and let $\epsilon$ denote the identity element of $\Gamma$.
For $\alpha,\beta\in\Gamma-\{\ep\}$, we write $\alpha\sim\beta$ if $r_M(E_{\{\alpha,\beta,\ep\}})=n$; then $\sim$ is an equivalence relation by Lemma \ref{equivalence}(i).
Let $\cA$ denote the set of equivalence classes under $\sim$; by hypothesis and Lemma \ref{equivalence}(ii) we have $|\cA|\ge 2$.
By Lemma \ref{generate}, this implies that $\Gamma$ is not cyclic.

We first use the fact that $\Gamma$ is abelian.

\begin{claim} \label{rel prime}
Let $\alpha, \beta \in \Gamma - \{\ep\}$. 
If there is a prime that divides the order of $\alpha$ but not the order of $\beta$, then $\alpha \sim \beta$. 
\end{claim}
\begin{proof}
By hypothesis, there are elements $\alpha',\beta'\in\Gamma$ with distinct prime orders so that $\alpha$ generates $\alpha'$, and $\beta$ generates $\beta'$.
Then $\alpha'\sim\alpha$, and $\beta'\sim\beta$, by Lemma \ref{generate}.
Since $\alpha'$ and $\beta'$ have distinct prime orders and $\Gamma$ is abelian, the subgroup of $\Gamma$ generated by $\{\alpha',\beta'\}$ is cyclic.
Thus, Lemma \ref{generate} implies that $\alpha'\sim\beta'$.
Since $\sim$ is an equivalence relation, this implies that $\alpha\sim\beta$.
\end{proof}


We now find the prime $p$.

\begin{claim} \label{prime}
There is a prime $p$ so that each element of $\Gamma$ has order $p$.
\end{claim}
\begin{proof}
By \ref{rel prime}, for any two elements of $\Gamma$ in different equivalence classes, there is a prime $p$ so that each has order equal to a power of $p$.
Since there are at least two equivalence classes of $\sim$, this implies that each element of $\Gamma$ has order equal to a power of $p$.

Now, let $\alpha$ be an element of order $p$.
We will show that each element in a different equivalence class has order $p$.
Since there are at least two equivalence classes, this implies that each element of $\Gamma$ has order $p$.
Let $\beta$ be in a different equivalence class than $\alpha$, and suppose that the order of $\beta$ is not $p$.
Then the subgroup of $\Gamma$ generated by $\{\alpha,\beta\}$ is isomorphic to $\bZ_p \oplus \bZ_{p^j}$ for some $j\ge 2$.
The elements $(0,1)$ and $(0,p)$ are in a common cyclic subgroup, as are the elements $(1,1)$ and $(0,p)$. 
By Lemma \ref{generate}, this implies that $(0,1) \sim (0,p) \sim (1,1)$, so $(0,1) \sim (1,1)$.
But the set $\{(0,1), (1,1)\}$ generates $\bZ_p \oplus \bZ_{p^j}$, so by Lemma \ref{equivalence}(iii), all elements of the subgroup generated by $\{\alpha, \beta\}$ are equivalent.
In particular, $\alpha \sim \beta$, a contradiction.
\end{proof}

Since $\Gamma$ is abelian and not cyclic, \ref{prime} implies that $\Gamma \cong \bZ_p^j$ for some $j\ge 2$.
\end{proof}


We conjecture that Theorem \ref{p-group} can be extended to all finite groups.

\begin{conjecture} \label{true?}
Let $n\ge 3$ be an integer, let $\Gamma$ be a finite group, and let $M \in \cM_{n,\Gamma}$.
If $r(M) - r(M(K_n^{\Gamma})) > 1$, then there is a prime $p$ and an integer $j\ge 2$ so that $\Gamma\cong \bZ_p^j$.
\end{conjecture}

We expect that Lemmas \ref{alpha}-\ref{equivalence} would all be useful for proving this conjecture.
Indeed, Lemma \ref{equivalence} implies that if $M \in \cM_{n, \Gamma}$ and $r(M) - r(M(K_n^{\Gamma})) > 1$, then $\Gamma$ has a nontrivial decomposition into subgroups with pairwise trivial intersection. 
This decomposition is a \emph{partition} of $\Gamma$, and the classification of finite groups that admit a nontrivial partition was completed by Baer, Kegel, and Suzuki \cite{Baer, Kegel, Suzuki}.
While Lemma \ref{equivalence} and this classification make partial progress towards Conjecture \ref{true?}, it is unclear how to proceed for non-abelian groups with a nontrivial partition.
\end{section}


\begin{section}{Examples}
Now that we have applied Theorem \ref{main} to group-labeled graphs, we consider other applications. 
Given a matroid $M$, it is unclear when there exists a matroid $N$ of rank at least three on the set of circuits of $M$ that satisfies $(\ref{construction}*)$ (or equivalently, $(\ref{main}*)$).
A rank-$1$ matroid $N$ on the circuits of $M$ can be constructed from any linear class $\cC$ of circuits of $M$, by taking $\cC$ to be precisely the set of loops of $N$.
Also, the rank-$2$ uniform matroid on the set of circuits of $M$ trivially satisfies $(\ref{construction}*)$, because the collection $\cC'$ is always spanning in $N$.
In this case, the matroid $M^N$ can also be obtained from applying Brylawski's construction (Theorem \ref{brylawski}) to $M$ with the empty linear class, and then again applying the construction with the empty linear class.
The following proposition shows that, for any matroid $M$ of corank at least three, there exists a rank-$3$ matroid $N$ on the set of circuits of $M$ that satisfies $(\ref{construction}*)$.

\begin{proposition} \label{rank-3}
Let $M$ be a matroid of corank at least three, and let $\cI$ be the collection of sets of circuits of $M$ so that $\cC \in \cI$ if and only if $|\cC| \le 3$ and each $\cC' \subseteq \cC$ satisfies $|\cC'| \le |\cup \cC'| - r_M(\cup \cC')$.
Then $\cI$ is the collection of independent sets of a rank-$3$ matroid $N$ on the set of circuits of $M$, and $N$ satisfies $(\ref{construction}*)$.
\end{proposition}
\begin{proof}
Clearly $\varnothing \in \cI$ and $\cI$ is closed under taking subsets.
Note that each $2$-element set $\{C_1, C_2\}$ of circuits of $M$ is in $\cI$, because $|C_1 \cup C| - r_M(C_1 \cup C) \ge 2$ since $M|(C_1 \cup C)$ contains distinct circuits.
Let $\cC_1$ and $\cC_2$ be sets in $\cI$ so that $|\cC_1| < |\cC_2|$. 
Assume that there is no circuit $C \in \cC_2$ for which $\cC_1 \cup \{C\} \in \cI$.
Since each $2$-element set of circuits is in $\cI$, we may assume that $|\cC_1| = 2$ and $|\cC_2| = 3$.

Let $C \in \cC_2 - \cC_1$. 
Then each $2$-element subset of $\cC_1 \cup \{C\}$ is in $\cI$.
Thus,
$$2 \le |\cup \cC_1| - r_M(\cup \cC_1) \le |\cup (\cC_1 \cup \{C\})| - r_M(\cup (\cC_1 \cup \{C\})) < 3,$$
where the first inequality holds because $\cC_1 \in \cI$, and the third holds because $\cC_1 \cup \{C\} \notin \cI$ and each subset is in $\cI$.
So equality holds throughout, which implies that $C \subseteq \cup \cC_1$.
The same reasoning applies to each circuit in $\cC_2 - \cC_1$, and so $\cup \cC_2 \subseteq \cup \cC_1$.
But then $$|\cup \cC_2| - r_M(\cup \cC_2) \le |\cup \cC_1| - r_M(\cup \cC_1) = 2,$$
so $\cC_2 \notin \cI$, a contradiction.
Thus, $\cI$ is the collection of independent sets of a matroid $N$ of rank at most three.
Since $M$ has corank at least three, it contains three circuits $C_1, C_2, C_3$ such that none is contained in the union of the other two; then $\{C_1, C_2, C_3\} \in \cI$, so $r(N) = 3$.

We now show that $N$ satisfies $(\ref{construction}*)$.
Let $\cC'$ be a perfect collection of circuits of $M$, and let $C \notin \cC'$ be a circuit of $M$ contained in $\cup \cC'$.
We may assume that $r_N(\cC') = 2$, or else $(\ref{construction}*)$ trivially holds since $r(N) = 3$.
Since no circuit in $\cC'$ is contained in the union of the others, this implies that $|\cC'| = 2$.
Since $C \subseteq \cup \cC'$, we have 
$$|\cup (\cC' \cup \{C\})| - r_M(\cup (\cC' \cup \{C\})) = |\cup \cC'| - r_M(\cup \cC') = |\cC'| < |\cC' \cup \{C\}|,$$ 
and so $\cC' \cup \{C\} \notin \cI$.
Then $r_N(\cC' \cup \{C\}) = r_N(\cC')$, and so $C \in \cl_N(\cC')$, as desired.
\end{proof}


For representable matroids we can do much better, using the derived matroid of Longyear \cite{Longyear} and Oxley and Wang \cite{Oxley+Wang}.
Let $M$ be an $\bF$-representable matroid with ground set $E = \{e_1,e_2,\dots,e_m\}$ for some field $\bF$, and fix an $\bF$-representation $A$ of $M$ with column vectors $\varphi(e_1), \varphi(e_2), \dots, \varphi(e_m)$.
For each circuit $C$ of $M$, there is a vector $\textbf{c}_C = (c_1, c_2, \dots, c_m)$ in $\bF^m$ such that $\sum_{i=1}^m c_i \varphi(e_i) = 0$ and $c_i \ne 0$ if and only if $e_i \in C$; this vector is unique up to multiplying by a nonzero scalar.
Let $A'$ denote the matrix over $\bF$ with columns indexed by the circuits of $M$ so that the column vector of each circuit $C$ is $\textbf{c}_C$.
Then $M(A')$ is the \emph{derived matroid} of the representation $A$ of $M$.
Oxley and Wang show that the rank of the derived matroid of any representation of $M$ is $r(M^*)$ \cite[Prop. 9.2.2]{Oxley+Wang}.
In addition, they show that the derived matroid of any representation of $U_{1,n}$ is $M(K_n)$ \cite[Lemma 2.5]{Oxley+Wang}, because each pair of elements in $[n]$ forms a circuit of $U_{1,n}$.
Note that $r(M(K_n)) = n-1 = r^*(U_{1,n})$.
We show that the derived matroid of any representation of a matroid $M$ satisfies $(\ref{construction}*)$.

\begin{proposition} \label{binary}
Let $M$ be matroid representable over a field $\bF$, and let $N$ be the derived matroid of a representation $A$ of $M$. Then $N$ satisfies $(\ref{construction}*)$, and $M^N$ is the free matroid on $E(M)$.
\end{proposition}
\begin{proof}
Let $\cC'$ be a perfect collection of circuits of $M$.
Since each circuit in $\cC'$ has an element that is not in any other circuit in $\cC'$, the set $\cC'$ is independent in $N$.
Let $N'$ denote the derived matroid of the matrix $A[(\cup \cC')]$; then $N'$ is a restriction of $N$, and 
$$r(N') = r((M|(\cup \cC'))^*) = |\cup \cC'| - r_M(\cup \cC') = |\cC'|,$$
where the last equality holds because $\cC'$ is perfect.
Since $\cC'$ is independent in $N$, this implies that $\cC'$ is a basis of $N'$.
Thus, each circuit $C$ of $M$ contained in $\cup \cC'$ satisfies $C \in \cl_N(\cC')$, and so $N$ satisfies $(\ref{construction}*)$.
Since $r(N) = r(M^*)$, Theorem \ref{construction} shows that $r(M^N) = r(M) + r(M^*) = |M|$, so $M^N$ is a free matroid.
\end{proof}

For example, the derived matroid $M(K_n)$ of any representation of $U_{1,n}$ satisfies $(\ref{construction}*)$, and $U_{1,n}^{M(K_n)}$ is the free matroid on $[n]$ because $r(U_{1,n}^{M(K_n)}) = r(U_{1,n}) + r(M(K_n)) = n$.

Given one matroid $N$ on the circuits of $M$ that satisfies $(\ref{construction}*)$, we can construct many more, using Proposition \ref{project N}.
If $M$ is representable, then Propositions \ref{binary} and \ref{project N} show that, for each integer $1 \le k < r^*(M)$, there are many rank-$k$ matroids on the circuits of $M$ that satisfy $(\ref{construction}*)$.
This proves Theorem \ref{representable}.
However, if $N$ and $N'$ are two different matroids on the circuits of $M$, it may be the case that $M^N = M^{N'}$.
For example, a matroid $M$ representable over a field $\bF$ other than $\GF(2)$ or $\GF(3)$ may have several non-isomorphic derived matroids, depending on the representation \cite[Theorem 2.6]{Oxley+Wang}, and these all lead to the free matroid, by Proposition \ref{binary}.
\end{section}


\begin{section}{The Converse}
Theorem \ref{brylawski} states that every elementary lift of a given matroid $M$ arises from a linear class of circuits of $M$.
More generally, the following restatement of Conjecture \ref{converse} states that every lift of $M$ arises from the construction of Theorem \ref{construction}.

\begin{conjecture} \label{converse2}
Let $M$ be a matroid. For every lift $K$ of $M$, there is a matroid $N$ on the circuits of $M$ so that $N$ satisfies $(\ref{construction}*)$, and $K\cong M^N$. 
\end{conjecture}

This is certainly true if $r(M) = 0$; then the set of circuits of $M$ is precisely $E(M)$, and we can take $N = K$.
It is also true if $M$ has corank at most two, by Theorem \ref{brylawski} and the fact that the rank-$2$ uniform matroid on the circuits of $M$  satisfies $(\ref{construction}*)$.
It is tedious but not difficult to check that it is true for certain small matroids of corank three, such as $U_{1,4}$ or $U_{2,5}$.
However, Conjecture \ref{converse2} seems difficult to prove even for the very basic class of rank-$1$ uniform matroids.

In general, if there exists a matroid $N$ so that $M^N$ is the free matroid on $E(M)$, then Proposition \ref{project N} shows that Theorem \ref{construction} can be used to construct a potentially huge number of lifts of $M$, and this provides some evidence that Conjecture \ref{converse2} is true for $M$.
In particular, the constructions of the previous section provide evidence that this conjecture is true for representable matroids and matroids of corank three.

One way to prove Conjecture \ref{converse2} would be to explicitly construct the matroid $N$, given $K$ and $M$.
We make the following conjecture in this direction:

\begin{conjecture} \label{matroid}
Let $M$ be a matroid, and let $K$ be a lift of $M$. 
Let $\cI$ be the collection of subsets $\cC'$ of $\cC(M)$ for which there is no matroid $K'$ such that 
\begin{itemize}
\item $K'$ is a projection of $K$ and a lift of $M$, 

\item each set in $\cC'$ is a circuit of $K'$, and 

\item $r(K)-r(K')<|\cC'|$.
\end{itemize}
Then $\cI$ is the collection of independent sets of a matroid $N$ on $\cC(M)$, and $K\cong M^N$.
\end{conjecture}

If $K$ is the free matroid on $E(M)$, then $K'$ is only required to be a lift of $M$, since every matroid on $E(M)$ is a projection of the free matroid on $E(M)$.
Conjecture \ref{matroid} may be easier to prove in this special case.
\end{section}


\begin{section}{The Dual Construction}
All of the previous results about lifts of a matroid $M$ give rise to results about projections of $M^*$.
For a set $E$ and a set $\cX$ of subsets of $E$, we write $\cap\cX$ for $\cap_{X\in\cX}X$.
Given a matroid $K$, we write $\cH(K)$ for the set of hyperplanes of $K$.
We say that a collection $\cH'$ of hyperplanes of $K$ is \emph{perfect} if $r_K(\cap \cH')=r(K)-|\cH'|$, and no hyperplane in $\cH'$ contains the intersection of the others.
It is easy to check that $\cH'$ is a perfect collection of hyperplanes of $K$ if and only if $\{E-H\colon H\in\cH'\}$ is a perfect collection of circuits of $K^*$.
The following is a restatement of  Theorem \ref{dual main}.
We omit the proof, as it is a straightforward application of duality.


\begin{theorem} \label{dual construction}
Let $K$ be a matroid, and let $N$ be a matroid on the set of hyperplanes of $K$ so that
\begin{enumerate}[$(\ref{dual construction}*)$]
\item if $\cH'$ is a perfect collection of hyperplanes of $K$, and $H$ is a hyperplane of $K$ that contains $\cap\cH'$, then $H\in \cl_N(\cH').$
\end{enumerate}
Then the function $r_{K_N}$ defined, for all $X \subseteq E(K)$, by 
$$r_{K_N}(X)=r_K(X)-r(N)+r_N(\{H\in \cH(K)\colon X\subseteq H\})$$
is the rank function of a rank-$r(N)$ projection $K_N$ of $K$.
\end{theorem}

Theorem \ref{crapo} states that every elementary projection of a given matroid $K$ arises from a linear subclass of hyperplanes of $K$.
More generally, we conjecture that every projection of $K$ arises from the construction of Theorem \ref{dual construction}; this is dual to Conjecture \ref{converse2}.
We close by stating the dual of Conjecture \ref{matroid}.

\begin{conjecture} \label{dual matroid}
Let $K$ be a matroid, and let $M$ be a projection of $K$. 
Let $\cI$ denote the collection of subsets $\cH'$ of $\cH(K)$ for which there is no matroid $K'$ such that 
\begin{itemize}
\item $K'$ is a projection of $K$ and a lift of $M$, 

\item each set in $\cH'$ is a hyperplane of $K'$, and 

\item $r(K')-r(M)<|\cH'|$.
\end{itemize}
Then $\cI$ is the collection of independent sets of a matroid $N$ on $\cH(K)$, and $M \cong K_N$.
\end{conjecture}

\end{section}

\begin{section}*{Acknowledgements}
The author would like to thank James Oxley for his helpful comments on the manuscript.
\end{section}


\end{document}